\documentclass[11pt]{amsart}

\usepackage{amsmath}
\usepackage{amsthm}
\usepackage{amsfonts}
\usepackage{xcolor}
\usepackage[english]{babel}
\usepackage[margin=1.0in]{geometry}
\usepackage[colorlinks=true,
        linkcolor=blue]{hyperref}
\usepackage{bbm}
\usepackage{verbatim}
\usepackage[abbrev]{amsrefs}

\numberwithin{equation}{section}

\newtheorem{prop}{Proposition}
\newtheorem{lemma}[prop]{Lemma}

\newtheorem{thm}[prop]{Theorem}
\newtheorem{cor}[prop]{Corollary}

\numberwithin{prop}{section}
\newtheorem*{acknow}{Acknowledgment}

\theoremstyle{definition}

\newtheorem{rmk}[prop]{Remark}

\definecolor{c1}{rgb}{0.2,0.4,0.5}
\definecolor{c2}{rgb}{0.1,0.3,0.5}
\definecolor{c3}{rgb}{0.2,0.7,0.5}
\usepackage{tikz}

\newcommand{\nab}{\nabla}

\newcommand{\abs}[1]{\left\lvert#1\right\rvert}

\DeclareMathOperator{\Ric}{Ric}

\title[Remarks on Entropy Formulae for Linear Heat Equation]{Remarks on Entropy Formulae for Linear Heat Equation}
\date{July 20th, 2022}
\author[Ji]{Yucheng Ji}
\email{\href{mailto:jiyucheng1991@gmail.com}{jiyucheng1991@gmail.com}}
\address{Shanghai Jiao Tong University, 800 Dongchuan RD, Minhang District, Shanghai 200240, China}\address{Huawei Technologies, No.2222 New Jinqiao RD, Pudong New District, Shanghai 201206, China
}

\begin{document}
\maketitle

\begin{abstract}
In this note, we prove some new entropy formula for linear heat equation on static Riemannian manifold with nonnegative Ricci curvature. The results are analogies of Cao and Hamilton's entropies for Ricci flow coupled with heat-type equations.
\end{abstract}

\section{Introduction} 
The study of differential Harnack inequalities of heat-type equations on manifolds was initiated from a seminal paper by P. Li and S.-T. Yau \cite{MR834612}, where they proved a gradient estimate for positive solutions of heat equation on Riemannian manifolds with nonnegative Ricci curvature. By integrating this inequality along space-time paths, they also obtained a sharp form of Harnack inequality. Later on, their ideas was adapted by R. Hamilton to the case of Ricci flow: he first proved an analogy of Li-Yau's result on surfaces \cite{MR954419} and then got a much more complicated, matrix version in the high-dimensional case \cite{MR1198607}. Around the same time, similar estimate was also obtained by H.-D. Cao for K\"ahler-Ricci flow \cite{MR1172691}. More results along this line could be found in \cite{MR2488944} and references therein. This approach to the differential inequalities which yield Harnack type estimates was named as Li-Yau-Hamilton by L. Ni and L.-F. Tam \cite{MR1981036}.

Another breakthrough in Ricci flow was the Li-Yau type inequality of conjugate heat equation (coupled with Ricci flow), and related entropy formula along with reduced volume monotonicity, discovered by G. Perelman \cite{P}. His results finally led him to the proof of Poincar\'e conjecture and Geometrization conjecture. Perelman's beautiful theorems and their powerful applications inspire people to find other possible formulae of Ricci flow (maybe coupled with heat equation or conjugate heat equation). In \cite{MR2433961}\cite{MR2570311}, X. Cao and R. Hamilton developed a unified way to find all possible (scalar) Li-Yau type inequalities under these settings, namely, interpolation of all possible terms in differential Harnack quantities. To illustrate the difference between their work and Perelman's, we shall recall that Perelman's entropy is for the fundamental type (heat-kernel type) solution of conjugate heat equation (under Ricci flow); in contrast, Cao and Hamilton considered both forward and backward heat equations with potentials, and general positive solutions instead of fundamental type solutions of those heat equations.

On the other hand, it seems to be a general principle that every differential Harnack estimate for Ricci flow would have a counterpart for heat equation with static metric. Matrix Li-Yau-Hamilton inequality for linear heat equation was proved by Hamilton himself in \cite{MR1230276} ; the corresponding result in the K\"ahler case was settled by H.-D. Cao and L. Ni in \cite{MR2148797}. After Perelman's breakthrough, L. Ni \cite{MR2030576}\cite{MR2051693} found an entropy formula for linear heat equation on Riemannian manifold with static metric, and also obtained some geometric applications. His result can be seen as `the linear version' of Perelman's entropy. L. Ni and L.-F. Tam \cite{MR1981036}\cite{MR2032112} also developed such theorems for heat equation on $(1,1)$-forms, L. Ni and Y. Niu \cite{MR2828586} on $(p,p)$-forms on K\"ahler manifolds.

In this note, we will combine the ideas of Ni \cite{MR2030576} and of Cao and Hamilton \cite{MR2433961}\cite{MR2570311}. Namely, we are going to prove an entropy formula for linear heat equation which is the counterpart of entropy formulae for Ricci flow proved by Cao and Hamilton, just as Ni's entropy is the counterpart of Perelman's entropy in the case of linear heat equation on static manifold. We will prove:

\begin{thm}\label{thm1.1}
Let $(M,g)$ be a closed Riemannian manifold of dimension $n$ and with nonnegative Ricci curvature.
Let $f$ be a positive solution to the heat equation
\begin{equation}\label{eq1.1}
\frac{\partial f}{\partial t}=\Delta f.
\end{equation}
Here the Laplacian is defined as $g^{ij}\partial_{i}\partial_{j}$. Let $f=e^{-u}$, and we define
\begin{equation*}
F=\int_{M}(t^2\abs{\nab u}^2-2nt)e^{-u}dV_{g},
\end{equation*}
then for all time $t>0$,
\begin{equation*}
F\leq0
\end{equation*}
and
\begin{equation*}
\frac{dF}{dt}\leq0.
\end{equation*}
\end{thm}

We will reduce the proof of entropy formula to the proof of the following differential Harnack inequality:
\begin{prop}\label{prop1.2}
Let $(M,g)$ be a closed Riemannian manifold of dimension $n$ with nonnegative Ricci curvature.
Let $f$ be a positive solution to the heat equation (\ref{eq1.1}), $u=-\ln f$ and
\begin{equation*}
H=2\Delta{u}-{\abs{\nab{u}}}^2-\frac{2n}{t}.
\end{equation*}
Then for all time $t>0$,
\begin{equation*}
H\leq0.
\end{equation*}
\end{prop}

\begin{rmk}
Compare Theorem \ref{thm1.1} (and Proposition \ref{prop1.2}) here with Theorem 4.1 (and Theorem 1.1) in \cite{MR2570311}, we can easily see that these are `linear version' of those:

\begin{prop}(Theorem 1.1 in \cite{MR2570311})\label{prop1.4}
Let $(M,g(t))$, $t \in [0, T)$, be a solution to the Ricci flow 
\begin{equation*}
\frac{\partial g_{ij}}{\partial t}=-2R_{ij}
\end{equation*}
on a closed Riemannian manifold of dimension $n$, and suppose that $g(t)$ has nonnegative curvature operator.
Let $f$ be a positive solution to the heat equation (\ref{eq1.1}), $u=-\ln f$ and
\begin{equation*}
H=2\Delta{u}-{\abs{\nab{u}}}^2-3R-\frac{2n}{t}.
\end{equation*}
Then for all time $t \in (0,T)$,
\begin{equation*}
H\leq0.
\end{equation*}
\end{prop}

\begin{thm}(Theorem 4.1 in \cite{MR2570311})
Assume the same conditions as in Proposition \ref{prop1.4}. Let
\begin{equation*}
F=\int_{M}t^2He^{-u}dV_{g},
\end{equation*}
then for all $t\in(0, T)$, we have $F\leq0$ and 
\begin{equation*}
\frac{dF}{dt}\leq0.
\end{equation*}
\end{thm}
\end{rmk}

The rest of this note is organized as follows. In Section \ref{sec2}, we will provide the proof of Proposition \ref{prop1.2}. An integral version of the Harnack inequality (Theorem \ref{thm2.3}) will also be given. In Section \ref{sec3}, we will discuss about Ni's entropy. Finally, in Section \ref{sec4}, the proof of Theorem \ref{thm1.1} will be completed. 

\begin{acknow}
The author would like to thank his PhD advisor, Prof. Zhiqin Lu for his constant support and many helpful conversations.
\end{acknow}

\section{Proof of Proposition \ref{prop1.2}.} \label{sec2}
In this section we will prove the differential Harnack estimate Proposition \ref{prop1.2}. The strategy is similar as in  \cite{MR2433961} and \cite{MR2570311}; we shall first derive a general evolution formula for function $H$.
Let us consider positive solutions of the heat equation (\ref{eq1.1}):
\begin{equation*}
\frac{\partial f}{\partial t}=\Delta f,
\end{equation*}
let $f=e^{-u}$, then $\ln f= -u$. We have
\begin{equation*}
\frac{\partial}{\partial t}\ln f=-\frac{\partial}{\partial t}u,
\end{equation*}
and
\begin{equation*}
\nab\ln f=-\nab u,\ \Delta\ln f=-\Delta u.
\end{equation*}
Hence $u$ satisfies the following equation:
\begin{equation}\label{eq2.1}
\frac{\partial u}{\partial t}=\Delta u-\abs{\nab u}^2.
\end{equation}

\begin{lemma}\label{lem2.1}
Let $(M,g)$ be a Riemannian manifold of dimension $n$ and $u$ be a solution to (\ref{eq2.1}). Let
\begin{equation*}
H=\alpha\Delta u-\beta\abs{\nab u}^2-b\frac{u}{t}-c\frac{n}{t}
\end{equation*}
where $\alpha,\beta,b$ and $c$ are constants that we will pick later. Then $H$ satisfies the following evolution equation:
\begin{align*}
\frac{\partial}{\partial t}H=\Delta H&-2\nab H\cdot\nab u-2(\alpha-\beta)\abs{\nab_{i}\nab_{j}u-\frac{\lambda g_{ij}}{2t}}^2-2(\alpha-\beta)\Ric(\nab u,\nab u)
\\
&-\frac{2(\alpha-\beta)}{\alpha}\frac{\lambda}{t}H-\big(b+\frac{2(\alpha-\beta)\beta}{\alpha}\lambda\big)\frac{{\abs{\nab u}}^2}{t}+\big(1-\frac{2(\alpha-\beta)}{\alpha}\lambda\big)b\frac{u}{t^2}
\\
&+\big(1-\frac{2(\alpha-\beta)}{\alpha}\lambda\big)c\frac{n}{t^2}+(\alpha-\beta)\frac{n{\lambda}^2}{2{t}^2},
\end{align*}
where $\lambda$ is also a constant that we will pick later.
\end{lemma}

\begin{proof}
The proof follows from a direct computation. We calculate the evolution equation of $H$ term by term. First, since $M$ has static metric,
\begin{align*}
\frac{\partial}{\partial t}(\Delta u)&=\Delta(\frac{\partial}{\partial t}u)
\\
&=\Delta(\Delta u)-\Delta(\abs{\nab u}^2).
\end{align*}

Secondly,
\begin{align*}
\frac{\partial}{\partial t}\abs{\nab u}^2=&\frac{\partial}{\partial t}(g^{ij}\nab_{i}u \nab_{j}u)
\\
=&g^{ij}(\nab_{i}\frac{\partial u}{\partial t})\nab_{j}u+g^{ij}\nab_{i}u(\nab_{j}\frac{\partial u}{\partial t})
\\
=&2\nab(\Delta u)\cdot\nab u-2\nab(\abs{\nab u}^2)\cdot\nab u;
\end{align*}
since
\begin{align*}
\Delta(\abs{\nab u}^2)=&g^{ij}\nab_{i}\nab_{j}(g^{\alpha\beta}{\nab_{\alpha}u}{\nab_{\beta}u})
\\
=&g^{\alpha\beta}\big(g^{ij}{\nab}_{i}{\nab}_{j}({\nab}_{\alpha}u)\big){\nab}_{\beta}u+g^{\alpha\beta}{\nab}_{\alpha}u\big(g^{ij}{\nab}_{i}{\nab}_{j}({\nab}_{\beta}u)\big)
\\
&+g^{ij}g^{\alpha\beta}({\nab}_{i}{\nab}_{\alpha}u)({\nab}_{j}{\nab}_{\beta}u)+
g^{ij}g^{\alpha\beta}({\nab}_{j}{\nab}_{\alpha}u)({\nab}_{i}{\nab}_{\beta}u)
\\
=&2\Delta(\nab u)\cdot\nab u+2\abs{\nab\nab u}^2,
\end{align*}
and by Ricci identity
\begin{equation*}
\Delta(\nab u)\cdot\nab u=\nab(\Delta u)\cdot\nab u+\Ric(\nab u,\nab u);
\end{equation*}
so we get that
\begin{align*}
\frac{\partial}{\partial t}\abs{\nab u}^2=&\Delta(\abs{\nab u}^2)-2\abs{\nab\nab u}^2-2\nab(\abs{\nab u}^2)\cdot\nab u-2\Ric(\nab u,\nab u).
\end{align*}

Combine with (\ref{eq2.1}), we arrive at
\begin{align*}
\frac{\partial}{\partial t}H=&\Delta H-\alpha\Delta({\abs{\nab u}}^2)+2\beta\abs{\nab\nab u}^2+2\beta\nab(\abs{\nab u}^2)\cdot\nab u
\\
&+2\beta \Ric(\nab u,\nab u)+b\frac{\abs{\nab u}^2}{t}+b\frac{u}{t^2}+c\frac{n}{t^2}
\\
=&\Delta H-2\nab H\cdot\nab u-2(\alpha-\beta)\abs{\nab\nab u}^2
\\
&-2(\alpha-\beta)\Ric(\nab u,\nab u)-b\frac{\abs{\nab u}^2}{t}+b\frac{u}{t^2}+c\frac{n}{t^2}
\\
=&\Delta H-2\nab H\cdot\nab u-2(\alpha-\beta)\abs{\nab_{i}\nab_{j}u-\frac{\lambda}{2t}g_{ij}}^2
\\
&-2(\alpha-\beta)\Ric(\nab u,\nab u)-2(\alpha-\beta)\frac{\lambda}{t}\Delta u
\\
&+(\alpha-\beta)\frac{n{\lambda}^2}{2t^2}-b\frac{\abs{\nab u}^2}{t}+b\frac{u}{t^2}+c\frac{n}{t^2}
\\
=&\Delta H-2\nab H\cdot\nab u-2(\alpha-\beta)\abs{\nab_{i}\nab_{j}u-\frac{\lambda}{2t}g_{ij}}^2
\\
&-2(\alpha-\beta)\Ric(\nab u,\nab u)-\frac{2(\alpha-\beta)}{\alpha}\frac{\lambda}{t}H
\\
&-\Big(b+\frac{2(\alpha-\beta)}{\alpha}\lambda\beta\Big)\frac{\abs{\nab u}^2}{t}+\Big(1-\frac{2(\alpha-\beta)}{\alpha}\lambda\Big)b\frac{u}{t^2}
\\
&+(\alpha-\beta)\frac{n{\lambda}^2}{2t^2}+\Big(1-\frac{2(\alpha-\beta)}{\alpha}\lambda\Big)c\frac{n}{t^2}.         
\end{align*}
\end{proof}

In the above lemma, let us take $\alpha=2$, $\beta=1$, $b=0$, $c=2$, $\lambda=2$. As a consequence, we have
\begin{cor}\label{cor2.2}
Let $(M,g)$ be a Riemannian manifold of dimension $n$, $f$ be a positive solution to the heat equation (\ref{eq1.1}); $u=-\ln f$ and 
\begin{equation}\label{eq2.2}
H=2\Delta u-\abs{\nab u}^2-\frac{2n}{t},
\end{equation}
then we have
\begin{align*}
\frac{\partial}{\partial t}H=\Delta H-2\nab H\cdot\nab u-2\abs{\nab_{i}\nab_{j}u-\frac{\lambda g_{ij}}{2t}}^2-2\Ric(\nab u,\nab u)-\frac{2H}{t}-2\frac{{\abs{\nab u}}^2}{t}.
\end{align*}

\end{cor}
Now we can finish the proof of Proposition \ref{prop1.2}.

\begin{proof}[Proof of Proposition \ref{prop1.2}] 
Since $M$ is a closed manifold, it is easy to see that for $t$ small enough, $H(t)<0$.
Then by Corollary \ref{cor2.2} and the maximum principle of heat equation,
\begin{equation*}
H\leq0
\end{equation*}
for all time $t$, provided $M$ has nonnegative Ricci curvature.
\end{proof} 

We can now integrate the inequality along a space-time path, and have the following:
\begin{thm}\label{thm2.3}
Let $(M,g)$ be a Riemannian manifold of dimension $n$ and with nonnegative Ricci curvature. Let $f$ be a positive solution to the heat equation (\ref{eq1.1}).
Assume that $(x_{1},t_{1})$ and $(x_{2},t_{2})$, with $t_{2}>t_{1}>0 $, are two points in $M\times(0,\infty)$. Let
\begin{equation*}
\Gamma=\inf_{\gamma}\int^{t_{2}}_{t_{1}}\abs{\dot \gamma}^2dt,
\end{equation*}
where $\gamma$ is any space-time path joining $(x_{1},t_{1})$ and $(x_{2},t_{2})$. Then we have
\begin{equation*}
f(x_{1}, t_{1})\leq f(x_{2}, t_{2})\Big(\frac{t_{2}}{t_{1}}\Big)^{n}e^{\frac{\Gamma}{2}}.
\end{equation*}
\end{thm}

\begin{proof} 
Since $H\leq0$ and $u$ satisfies (\ref{eq2.1}):
\begin{equation*}
\frac{\partial u}{\partial t}=\Delta u-\abs{\nab u}^2,
\end{equation*}
we have
\begin{equation*}
2\frac{\partial u}{\partial t}+\abs{\nab u}^2-\frac{2n}{t}\leq 0.
\end{equation*}
If we pick a space-time path $\gamma(x,t)$ joining $(x_{1},t_{1})$ and  $(x_{2},t_{2})$ with $t_{2}>t_{1}> 0$, then along $\gamma$, we get
\begin{align*}
\frac{du}{dt}&=\frac{\partial u}{\partial t}+\nab u\cdot\dot \gamma
\\
&\leq -\frac{1}{2}\abs{\nab u}^2+\frac{n}{t}+\nab u\cdot\dot \gamma
\\
&\leq\frac{1}{2}\abs{\dot \gamma}^2+\frac{n}{t},
\end{align*}
thus
\begin{equation*}
u(x_{2},t_{2})-u(x_{1},t_{1})\leq\frac{1}{2}\inf_{\gamma}\int^{t_{2}}_{t_{1}}\abs{\dot \gamma}^2dt+n\ln\Big(\frac{t_{2}}{t_{1}}\Big).
\end{equation*}
By the definition of $u$ and $\Gamma$ we obtain
\begin{equation*}
f(x_{1}, t_{1})\leq f(x_{2}, t_{2})\Big(\frac{t_{2}}{t_{1}}\Big)^{n}e^{\frac{\Gamma}{2}}.
\end{equation*}
\end{proof} 

\section{About Fundamental Solutions of Heat Equation}\label{sec3}

In this section, we consider the fundamental solutions of heat equation (\ref{eq1.1}). Let $f=(4\pi t)^{-\frac{n}{2}}e^{-v}$, then $v=-\ln f -\frac{n}{2}\ln(4\pi t)$. We have
\begin{equation*}
\frac{\partial}{\partial t}\ln f=-\frac{\partial}{\partial t}v-\frac{n}{2t},
\end{equation*}
and
\begin{equation*}
\nab\ln f=-\nab v,\ \Delta\ln f=-\Delta v.
\end{equation*}
Hence $v$ satisfies the following equation:
\begin{equation}\label{eq3.1}
\frac{\partial v}{\partial t}=\Delta v-\abs{\nab v}^2-\frac{n}{2t}.
\end{equation}

\begin{lemma}\label{lem3.1}
Let $(M,g)$ be a Riemannian manifold and $v$ be a solution to (\ref{eq3.1}). Let
\begin{equation*}
P=\alpha\Delta v-\beta\abs{\nab v}^2-b\frac{v}{t}-c\frac{n}{t}
\end{equation*}
where $\alpha,\beta,a,b$ and $d$ are constants that we will pick later. Then $P$ satisfies the following evolution equation:
\begin{align*}
\frac{\partial}{\partial t}P=\Delta P&-2\nab P\cdot\nab v-2(\alpha-\beta)\abs{\nab_{i}\nab_{j}v-\frac{\lambda g_{ij}}{2t}}^2-2(\alpha-\beta)\Ric(\nab v,\nab v)
\\
&-\frac{2(\alpha-\beta)}{\alpha}\frac{\lambda}{t}P-\big(b+\frac{2(\alpha-\beta)\beta}{\alpha}\lambda\big)\frac{{\abs{\nab v}}^2}{t}+\big(1-\frac{2(\alpha-\beta)}{\alpha}\lambda\big)b\frac{v}{t^2}
\\
&+\big(1-\frac{2(\alpha-\beta)}{\alpha}\lambda\big)c\frac{n}{t^2}+(\alpha-\beta)\frac{n{\lambda}^2}{2{t}^2}+b\frac{n}{2t^2},
\end{align*}
where $\lambda$ is also a constant that we will pick later.
\end{lemma}

\begin{proof}
The proof again follows from the same direct computation as in the proof of Lemma \ref{lem2.1}. Notice that the only extra term $b\frac{n}{2t^2}$ comes from the evolution of $-b\frac{v}{t}$.
\end{proof}

Now again we analyze all the terms in the equation here. To apply the maximum principle of linear heat operator, we must make sure all the terms after $\Delta P$ are either multiple of $P$, first order variation of $P$, or some terms with definite negative sign. Obviously we cannot control the sign of the term $\big(1-\frac{2(\alpha-\beta)}{\alpha}\lambda\big)b\frac{v}{t^2}$ (because $v$ can have indefinite sign on $M$), so we must kill this term. Then we have two options:
\\
\\
(1) $1-\frac{2(\alpha-\beta)}{\alpha}\lambda=0$. 
\\
\\
In this case, we have the evolution equation becomes 
\begin{align*}
\frac{\partial}{\partial t}P=\Delta P&-2\nab P\cdot\nab v-2(\alpha-\beta)\abs{\nab_{i}\nab_{j}v-\frac{\lambda g_{ij}}{2t}}^2-2(\alpha-\beta)\Ric(\nab v,\nab v)
\\
&-\frac{P}{t}-(b+\beta)\frac{{\abs{\nab v}}^2}{t}+\frac{{\alpha}^2}{4(\alpha-\beta)}\frac{n}{2{t}^2}+b\frac{n}{2t^2}.
\end{align*}
\\
So obviously we must have $\alpha-\beta\geq 0$, $b+\beta\geq 0$ and $\frac{{\alpha}^2}{4(\alpha-\beta)}+b\leq 0$. Then $\frac{{\alpha}^2}{4(\alpha-\beta)}-\beta\leq 0$, which is just $\frac{(\alpha-2\beta)^2}{4(\alpha-\beta)}\leq 0$. This forces us to choose $\alpha=2\beta=-2b>0$, $\lambda=1$. Now $c$ is the only constant needs to be determined. Since $P$ should be negative at $t=0$, any strictly positive number would qualify for $c$. Let us just set $c=-b$. Hence after a rescaling, we recover the entropy of L. Ni (see \cite{MR2030576}, Theorem 1.2):
$$
P=2\Delta v-\abs{\nab v}^2+\frac{v}{t}-\frac{n}{t}.
$$
From the argument above, we can see that Ni's entropy is the unique entropy of heat equation in this case. In fact, it is not only the counterpart of Perelman's entropy, but also the counterpart of one of Cao and Hamilton's entropies (see \cite{MR2570311}, Theorem 1.2), on the static manifold.
\\
\\
(2) $b=0$. 
\\
\\
Here the evolution equation in Lemma \ref{lem3.1} reads 
\begin{align*}
\frac{\partial}{\partial t}P=\Delta P&-2\nab P\cdot\nab v-2(\alpha-\beta)\abs{\nab_{i}\nab_{j}v-\frac{\lambda g_{ij}}{2t}}^2-2(\alpha-\beta)\Ric(\nab v,\nab v)
\\
&-\frac{2(\alpha-\beta)}{\alpha}\frac{\lambda}{t}P-\frac{2(\alpha-\beta)\beta}{\alpha}\lambda\frac{{\abs{\nab v}}^2}{t}+\big(1-\frac{2(\alpha-\beta)}{\alpha}\lambda\big)c\frac{n}{t^2}
\\
&+(\alpha-\beta)\frac{n{\lambda}^2}{2{t}^2}.
\end{align*}

In this case, we may have various differential Harnack inequalities. If we take $\alpha=2$, $\beta=1$, $c=2$, $\lambda=2$, then
\begin{cor}
Let $(M,g)$ be a Riemannian manifold of dimension $n$, $f$ be a positive fundamental solution to the heat equation; $v=-\ln f-\frac{n}{2}\ln(4\pi t)$ and 
\begin{equation}
P=2\Delta v-\abs{\nab v}^2-\frac{2n}{t},
\end{equation}
then we have
\begin{align*}
\frac{\partial}{\partial t}P=\Delta P&-2\nab P\cdot\nab v-2\abs{\nab_{i}\nab_{j}v-\frac{\lambda g_{ij}}{2t}}^2-2\Ric(\nab v,\nab v)-\frac{2P}{t}-2\frac{{\abs{\nab v}}^2}{t}.
\end{align*}

\end{cor}

However, $P$ is the same thing as $H$ in Proposition \ref{prop1.2}.
\\

On the other hand, we can also take $\alpha=2$, $\beta=0$, $c=1$, $\lambda=1$, then we indeed get the classical Li-Yau differential Harnack inequality (see \cite{MR834612}):

\begin{cor}[Li-Yau]
Let $M$ be a closed Riemannian manifold of dimension $n$ and with nonnegative Ricci curvature, $f$ be a positive fundamental solution to the heat equation; $v=-\ln f-\frac{n}{2}\ln(4\pi t)$ and 
\begin{equation}
P=2\Delta v-\frac{n}{t},
\end{equation}
then we have
\begin{align*}
\frac{\partial}{\partial t}P=\Delta P-2\nab P\cdot\nab v-4\abs{\nab_{i}\nab_{j}v-\frac{\lambda g_{ij}}{2t}}^2-4\Ric(\nab v,\nab v)-\frac{4P}{t}.
\end{align*}
So by the maximum principle, $P\leq0$ for all $t>0$.
\end{cor}

\section{Proof of Theorem \ref{thm1.1}.} \label{sec4}
In this section, we complete the proof of monotonicity of the entropy given in Theorem \ref{thm1.1}.

First, let $(M,g)$ be the closed Riemannian manifold, and $f$ be a positive solution of the heat equation (\ref{eq1.1}).
Let $u=-\ln f$ and $H$ as in (\ref{eq2.2}), we define
\begin{equation*}
F=\int_{M}t^2He^{-u}dV_{g},
\end{equation*}
it is easy to observe that
\begin{align*}
F&=\int_{M}t^2He^{-u}dV_{g}=\int_{M}t^2(2\Delta u-\abs{\nab u}^2-\frac{2n}{t})e^{-u}dV_{g}
\\
&=\int_{M}t^2(\abs{\nab u}^2-\frac{2n}{t})e^{-u}dV_{g}+\int_{M}t^2(2\Delta u-2\abs{\nab u}^2)e^{-u}dV_{g}
\\
&=\int_{M}t^2(\abs{\nab u}^2-\frac{2n}{t})e^{-u}dV_{g}-2t^2\int_{M}\Delta(e^{-u})dV_{g}
\\
&=\int_{M}(t^2\abs{\nab u}^2-2nt)e^{-u}dV_{g}.
\end{align*}
Here we used Stokes' theorem for the last equality. So $F$ is indeed the entropy defined in Theorem \ref{thm1.1}.

\begin{proof}[Proof of Theorem \ref{thm1.1}] 
$F\leq0$ follows directly from $H\leq0$. For the time derivative of $F$, using (\ref{eq2.1}) and Corollary \ref{cor2.2}
we have
\begin{align*}
\frac{d}{dt}F=\frac{d}{dt}&\int_{M}t^2He^{-u}dV_{g}
\\
=\int_{M}&\Big(2tHe^{-u}+t^2e^{-u}\frac{\partial}{\partial t}H+t^2H\frac{\partial}{\partial t}e^{-u}\Big)dV_{g}
\\
=\int_{M}&\Big(t^2e^{-u}\big(\Delta H-2\nab H\cdot\nab u-2\abs{\nab_{i}\nab_{j}u-\frac{\lambda g_{ij}}{2t}}^2-2\Ric(\nab u,\nab u)-2\frac{{\abs{\nab u}}^2}{t}\big)
\\
&+t^2H\Delta e^{-u}\Big)dV_{g}
\\
=\int_{M}&\Big(\Delta(t^2He^{-u})-2t^2e^{-u}\big(\abs{\nab_{i}\nab_{j}u-\frac{\lambda g_{ij}}{2t}}^2+\Ric(\nab u,\nab u)+\frac{{\abs{\nab u}}^2}{t}\big)\Big)dV_{g}
\\
=-2&t^2\int_{M}\Big(e^{-u}\big(\abs{\nab_{i}\nab_{j}u-\frac{\lambda g_{ij}}{2t}}^2+\Ric(\nab u,\nab u)+\frac{{\abs{\nab u}}^2}{t}\big)\Big)dV_{g}\leq0
\end{align*}
since $M$ has nonnegative Ricci curvature.
\end{proof}

\begin{rmk}
In fact, if we turn the time direction backward, then Proposition \ref{prop1.2} is also the counterpart of:
\begin{prop}(Theorem 1.1 in \cite{MR2433961})\label{prop4.2}
Let $(M, g(t))$, $t\in[0,T)$, be a solution to the Ricci flow
on a closed Riemannian manifold of dimension $n$, and $f$ be a positive solution to the backward heat-type equation 
\begin{equation*}
\frac{\partial f}{\partial t}=-\Delta f+2Rf,
\end{equation*}
with $\tau=T-t$, $u=-\ln f$, and 
\begin{equation*}
H=2\Delta{u}-{\abs{\nab{u}}}^2+2R-\frac{2n}{\tau}.
\end{equation*}
Then for all $t\in [0,T)$,
\begin{equation*}
H\leq0;
\end{equation*}
\end{prop}

and

\begin{prop}(Theorem 1.3 in \cite{MR2433961})\label{prop4.3}
Let $(M, g(t))$, $t\in[0,T)$, be a solution to the Ricci flow
on a closed Riemannian manifold of dimension $n$, and suppose that $g(t)$ has nonnegative scalar curvature. Let  $f$ be a positive solution to the conjugate heat equation 
\begin{equation*}
\frac{\partial f}{\partial t}=-\Delta f+Rf,
\end{equation*}
with $\tau=T-t$, $u=-\ln f$, and 
\begin{equation*}
H=2\Delta{u}-{\abs{\nab{u}}}^2+R-\frac{2n}{\tau}.
\end{equation*}
Then for all $t\in [0,T)$,
\begin{equation*}
H\leq0.
\end{equation*}
\end{prop}

Likewise, Theorem \ref{thm1.1} is the counterpart of 
\begin{thm}(Theorem 4.1 in \cite{MR2433961})
Assume the same conditions as in Proposition \ref{prop4.2}. Let
\begin{equation*}
F=\int_{M}{\tau}^2He^{-u}dV_{g},
\end{equation*}
then for all $t\in(0, T)$, we have $F\leq0$ and 
\begin{equation*}
\frac{dF}{dt}\geq0;
\end{equation*}
\end{thm}

and

\begin{thm}(Theorem 4.2 of \cite{MR2433961})
Assume the same conditions as in Proposition \ref{prop4.3}. Let
\begin{equation*}
F=\int_{M}{\tau}^2He^{-u}dV_{g},
\end{equation*}
then for all $t\in(0, T)$, we have $F\leq0$ and 
\begin{equation*}
\frac{dF}{dt}\geq0.
\end{equation*}
\end{thm}

It is interesting to see that though the entropies for heat equations under Ricci flow are different in the forward and the backward cases, they yet reduce to the same one in the static case. We also hope to see more geometric applications of this new monotonicity formula in the future.
\end{rmk}

\bigskip
\bibliographystyle{plain}


\end{document}